\documentclass[arxiv,reqno,twoside,a4paper,11pt]{amsart}
\usepackage{lhtools}
\usepackage{lhmath62,lhthm10-REAR}

\usepackage{upref,amsmath,amssymb,amsthm,mathrsfs}
\usepackage{xy, latexsym, tensor}
\usepackage{graphicx,stackengine,scalerel}
\usepackage[colorlinks=true,linkcolor=blue,citecolor=blue]{hyperref}

\usepackage[scaled]{helvet} 
\usepackage{courier} 
\usepackage[mathbf]{euler}
\usepackage{lhlocal} 

\usepackage{pgfplots}
\usetikzlibrary{arrows.meta}

\numberwithin{equation}{section}

\definecolor{qqwuqq}{rgb}{0,0,0}

\begin{document}

\date{\today}

\title[Curvature estimates for graphs in warped products]{Curvature estimates 
for graphs in warped product spaces}

\author[A. P. Barreto]{Alexandre Paiva Barreto} 
\address{Department of Mathematics, 
	Universidade Federal de S\~ao Carlos (UFSCar),
	Brazil}
\email{alexandre@dm.ufscar.br}

\author[F. A. Coswosck]{Fabiani A. Coswosck}
\address{Universidade Federal do Esp\'irito Santo (UFES), Brazil.}
\email{fcoswosck@gmail.com}

\author[L. Hartmann]{Luiz Hartmann}
\address{Department of Mathematics, 
	Universidade Federal de S\~ao Carlos (UFSCar),
	Brazil}
\email{hartmann@dm.ufscar.br}
\urladdr{http://www.dm.ufscar.br/profs/hartmann}

\thanks{A. P. Barreto  is partially supported by FAPESP: 2018/03721-4, 
Fabiani A. Coswosck is partially supported by CAPES: 88881.068165/2014-01 and 
Luiz Hartmann is partially supported by FAPESP:2018/23202-1 .}


\begin{abstract} 
We prove local and global upper estimates for the infimum of the mean 
curvature, the scalar curvature and the norm of the shape operator of
graphs 
in a warped product space. Using these estimates, we 
obtain some results on pseudo-hyperbolic spaces 
and space forms.
\end{abstract}

\maketitle

\tableofcontents

\section{Introduction}


The study of submanifolds in a Riemannian manifold is a central subject on 
Differential Geometry. In many results, global properties of a submanifold 
are 
drawn from hypothesis on its 
topology and curvature. Just to exemplify, Aleksandrov’s theorem \cite{ALX} 
states that the round spheres are the only compact embedded hypersurfaces 
with constant mean curvature in Euclidean space. For a sample of results of 
the same nature see for example 
\cite{KLO,NM,ROS,ROS2,Che,MAN,ARO,FAL,BFH} and references therein.

\medskip

In recent years, many results of the nature described in the previous 
paragraph have been obtained in the case of hypersurfaces of a warped product 
space $M_{\psi} \times R$. For example, Montiel \cite{Mon} obtained 
conditions for a compact hypersurface with constant mean curvature in such a 
space to be a slice (see \cite{AD} for generalizations of this result). 


\medskip

An important class of hypersurfaces in warped product spaces is that of 
the en\-ti\-re 
graphs. A breakthrough result concerning this class is the Bernstein Theorem, 
which states that the planes are the only minimal entire graphs in the 
Euclidean 
3-space. Many Bernstein-type theorems in warped 
product spaces have appeared in the late years ({\it e.g.} 
\cite{AD2,AD,ADRip,CdL,CCdL,ADR}). Using the 
Alexandrov's 
reflection method, Frensel \cite{KF} proved that 
the only graphs with constant mean curvature in the half-plane model of the 
hyperbolic space are the horospheres (see also \cite[Theorem A]{DCL}). Aquino 
and de 
Lima 
\cite{AL} obtained a Bernstein-type theorem on a particular warped product 
space with 
additional assumptions on the second fundamental form of the graph. 

\medskip

Unfortunately, determining the curvature of hypersurfaces in arbitrary spaces 
is a task that in general is either difficult or requires elaborated 
calculations.  For this reason, obtaining curvature estimates of 
hypersurfaces is all that can be done in many situations.
%
%
%
%
%
In 1955, 
Heinz \cite{HE} obtained estimates for the 
mean curvature $H$ and Gaussian curvature $K$ of a surface in $\mathbb{R}^3$ 
which is the graph of a smooth function defined on a open disc of radius $r$ 
in the plane. He showed that  
\begin{equation*}
\inf|H|\leq \frac{1}{r} \qquad \text{and} \qquad \inf|K|\leq 
\frac{3e^2}{r^2}. 
\end{equation*}
Later on, Chern \cite{CH} and Flanders \cite{FL}, independently, extended the 
above inequality of the mean curvature to higher dimensions. After that, it 
was generalized by 
Finn \cite{FI} to a broader class of domains of the plane and by Salavessa 
\cite{SAL1,SAL2} for graphs over Riemannian manifolds. 
Inspired by these works, 
Fontenele and the second author
\cite{FONT,FF} established estimates of curvature for graphs 
in the Riemannian product $M\times\mathbb{R}$, from which they deduced 
sharp estimates for the infimum of the mean curvature for graphs over 
a complete Riemannian manifold with Ricci curvature bounded below. 

\medskip

In this work we establish several results of the same nature for graphs in a 
warped product space contained or not in a slab. For 
example, 
Corollary \ref{Corol-MeanCurvature-Slice} 
provides a version of \cite[Theorem  
2.9]{AD} for graphs in the warped product 
$M_{\psi}\times\mathbb{R}$,  
in which instead of making assumption on the Ricci curvature of the graph we 
make assumption on the sectional curvature of the base $M$.
Theorem 
\ref{Theo-MeanCurvOmoriYau-sigma} provides an estimate for the mean curvature 
of a graph in a warped product space, similar in spirit to  
\cite[Proposition  
2.10]{AD}. 
In that estimate we do not assume that the graph is contained in a slab, 
which allows us to apply it for any graph in the 
space  $M_{\cosh 
t}\times\mathbb{R}$, provide that the sectional curvature of $M$ is bounded 
from below.

\medskip

Our approach is based in a known relation between the principal curvatures of 
a graph in a warped product space $M_\psi \times I$ and the principal 
curvatures of a related graph in the Riemannian product  $M\times J$.
%
It is worth to point out that the curvature of the slices have 
strong influence in the estimates presented here.

\medskip

This paper is organized as follows. In Section \ref{Sec-Preliminaries}, we 
fix 
the notation and present basic results that will be used in the entire work. 
In Section \ref{Sec-Est-Curv-Cont-Slab}, we present curvature estimates for 
graphs contained in a slab, and in Section \ref{Sec-Est-Curv--Not-Cont-Slab} 
for any graph. In Section 
\ref{Sec-Applications}, we use the estimates of 
the previous sections to obtain results on pseudo-hyperbolic spaces and space 
forms. For example, we show that the 
only entire 
graphs with constant mean curvature contained in a slab of 
$\mathbb{H}^{m+1}\equiv 
\mathbb{H}^m_{\cosh t} \times \R$ or $(\S^m)_{\sinh t}\times 
(0,\infty) \equiv \mathbb{H}^{m+1}$ are the slices (see Corollaries 
\ref{cor4} and \ref{Cor-Sphere-Hype}).

\section{Preliminary and auxiliary concepts}\label{Sec-Preliminaries}

In this section we will fix the notation and present general 
results about warped product spaces and graphs. Our results can be applied 
for 
$C^2$-functions, however, for simplicity, we assume that all manifolds, 
functions, etc., are smooth.

Let $\left(M^m, \langle\cdot,\cdot\rangle_M\right)$ and 
$\left( N^n, 
\langle\cdot,\cdot\rangle_N\right)$ be Riemannian manifolds 
and let $\pi_M:M\times N \rightarrow M$, $\pi_N: M\times N 
\rightarrow N$ be the projections maps over $M$ and $N$, respectively. 
Let $\psi:N\rightarrow (0,\infty)$ be a positive function.
The product $M\times N$ equipped with the Riemannian metric 
\begin{equation}\label{wp1}
\langle \cdot, \cdot \rangle_{\psi} := \left( 
\psi\circ\pi_N\right)^2\pi_M^{\ast}\langle \cdot, \cdot \rangle_M + 
\pi_N^{\ast}  \langle \cdot, \cdot \rangle_N, 
\end{equation}
is called \textit{warped product space} and denoted by $M_{\psi}\times N$.   
The relationship between the curvature tensors $\Re_\psi$, ${\Re}^{N}$ and 
${\Re}^{M}$, of $M_{\psi}\times N$, $N$ and the 
fiber $M$, respectively, are 
given by following proposition (see \cite[p. 210]{ON}). 

\begin{prop}\label{ecpprop2}
	 If $U,V,W \in \mathfrak{X}(M)$ and $X,Y,Z\in \mathfrak{X}(N)$ are 
	 vector 
	fields, then  
	\begin{enumerate}
		\item [(1)] $\Re_\psi(\widetilde{X}, \widetilde{Y})\widetilde{Z} = 
		\widetilde{{\Re}^{N}\left( X,Y\right)Z}$. 
		\item [(2)] $\Re_\psi(\widetilde{V}, \widetilde{X} )\widetilde{Y} = 
		-\displaystyle\frac{\textnormal{Hess 
		}\psi\left(X,Y\right)}{\psi}\hspace{0.1cm}\widetilde{V}$.
		\item [(3)] $\Re_\psi( \widetilde{X}, \widetilde{Y})\widetilde{V} = 
		\Re(\widetilde{V}, \widetilde{W} )\widetilde{X} = 0$. 
		\item [(4)] $\Re_\psi(\widetilde{X}, \widetilde{V})\widetilde{W} = 
		-\frac{\left\langle \widetilde{V}, 
		\widetilde{W}\right\rangle_\psi}{\psi}\widetilde{\nabla_X \nabla 
		\psi}$.
		\item [(5)] $\Re_\psi(\widetilde{V}, \widetilde{W} )\widetilde{U} = 
		\widetilde{{\Re}^M( V,W)U} - \displaystyle\frac{\langle\nabla 
		\psi, \nabla \psi \rangle_N}{\psi^2}\left\{\left\langle 
		\widetilde{W}, \widetilde{U}\right\rangle_\psi \widetilde{V} - 
		\left\langle\widetilde{V}, \widetilde{U} 
		\right\rangle_{\psi}\widetilde{W} 
		\right\}$,
	\end{enumerate}
	where $\widetilde{\cdot}$ represents the lift of the field to 
	$M_{\psi}\times N$. 
\end{prop}

\begin{remark}
Differently from \cite{ON}, we are using the following definition for the 
curvature tensor of a Riemannian 
manifold 
\begin{equation*}
\Re(X,Y)Z := \nabla_X \nabla_Y Z - \nabla_Y \nabla_X Z - \nabla_{[x,y]} Z. 
\end{equation*}
This choice justify the sign difference in some 
of the previous relations.
\end{remark}

In the 
special case where $N$ is an open interval $I \subset \R$,
the {\it slices} $M\times\{t\}$  
are totally umbilical hypersurfaces with constant mean 
curvature (see \cite[p.206]{ON})
\begin{equation}\label{Eq-MeanCurvSlice}
	\mathcal{H}(t) := (\ln \psi)'(t)=\frac{\psi'}{\psi}(t),
\end{equation}
with respect to the vector field $-\partial_t$, where $\partial_t$ is  
the canonical unit normal vector field. 
By Proposition \ref{ecpprop2} item (5), we 
obtain that 
\begin{equation}\label{Eq-TangSecCurv}
	K_{M_\psi \times I}(p,t)(\tilde{u},\tilde{v}) = 
	\frac{K_M(p)(u,v)}{\psi^2(t)}-\H^2(t),
\end{equation}
where $u,v \in T_pM$ and $K_M$ is the sectional curvature of $M$. By 
\Eqref{Eq-TangSecCurv} and the Gauss 
equation, the sectional 
curvature of a 
slice at a point $(p,t)$, in the planed spanned by $\tilde{u}$, $\tilde{v}$, 
is given by
\begin{equation*}
\cK(p,t)(\tilde{u},\tilde{v}):= 
\frac{K_M(p)(u,v)}{\psi^2(t)}.
\end{equation*}
Again by Proposition \ref{ecpprop2}, we have that the 
sectional curvature $K_{M_\psi 
\times 
I}(\tilde{u},\partial_t)$ does not depends on $u\in T_p M$ and $p\in M$. 
Therefore
we define the {\it normal sectional 
curvature} of $M\times \{t\}$ by 
\begin{equation}\label{Eq-NormalSecCurv}
\cK^{\perp}(t):=K_{M_\psi \times I}(\tilde{u},\partial_t) = 
-\frac{\psi''}{\psi}(t).
\end{equation}
It follows from Proposition \ref{ecpprop2}, \Eqref{Eq-TangSecCurv}  and 
\eqref{Eq-NormalSecCurv} that, 
$M_{\psi}\times I$ has 
constant sectional curvature $\kappa$ if, and only if, $M$ has constant 
sectional curvature $K_M$, the normal sectional curvature is constant equal 
to $\kappa$ and 
$\H^2 + \kappa = \cK$. 

The main object studied here is the graph of a function $f:M\rightarrow 
I\subset \mathbb{R}$. The {\it graph} of 
$f$ is 
defined by  
\begin{equation*}
	\Gamma_f := \{\left( x, f(x)\right): x\in M\}. 
\end{equation*} 
The {\it canonical parametrization} of the graph of $f$ will be denoted by 
$\phi:M\rightarrow \Gamma_f\subset M_{\psi} \times I$, where $\phi(x) 
:= \left( 
x,f(x)\right)$. It is easy to see that  
\begin{equation}\label{wp3}
d\phi_x(v) = \widetilde{v} + \langle \nabla f(x),v\rangle_M \cdot
\partial_t, \hspace{0.4cm} \forall x\in M, \hspace{0.1cm} v\in T_x M.
\end{equation}
We choose the following {\it unit normal vector field to the graph} of $f$ 
in 
$M_{\psi}\times I$ 
\begin{equation}\label{wp4}
\eta\left( \phi(x)\right) = \frac{\psi\left( f(x)\right)}{\cW(x)}\left( 
\frac{\widetilde{\nabla 
f(x)}}{\psi^2(f(x))} - \partial_t \left( \phi(x)\right)\right),
\end{equation}
where $\cW(x):= \sqrt{\|\nabla f(x)\|^2+(\psi\circ f)^2(x) }$ and 
$\nabla f$ is the gradient vector field in $M$, \cf \cite[Eq.(3.3)]{AGH}.

The shape operator of a hypersurface in $M_\psi\times I$ was described in 
\cite[Eq.(2.14)]{AGH}. For the convenience of the reader, we state below this
formula according our notations. 
For every tangent vector  $v,w\in T_x M$ and $x\in M$, {\it the shape 
operator of the 
graph} of 
$f$ in $M_{\psi}\times I$, with respect to $\eta$, satisfies 
\begin{equation}\label{wpprop01}
\begin{aligned}
\left\langle A\left( d\phi_x(v)\right), d\phi_x(w)\right\rangle_{\psi} &=& 
-\frac{\psi\left( f(x)\right)}{\cW(x)} \textnormal{Hess}f_x(v,w) 
+ 
\frac{(\psi^3 \cdot\H)(f(x))}{\cW(x)}\langle 
v,w\rangle_M\\
& & + 2\frac{(\psi\cdot  \H)\left( f(x)\right)}{\cW(x)} 
\langle\nabla 
f(x),v\rangle_M \cdot \langle\nabla f(x), w\rangle_M.
\end{aligned}
\end{equation}

The following proposition is consequence from the last equation.

\begin{prop}\label{wpprop1}
	The mean curvature of $\Gamma_f$ in 
	$M_{\psi}\times I$ with respect to 
		the unit normal vector field $\eta$ is given by  
		\begin{equation*}
		\begin{aligned}
		mH\left( \phi(x)\right) &= -\textnormal{Div}_M \left(\frac{\nabla 
		f}{\psi(f)\cdot\cW}\right)(x)
		+\frac{(\psi\cdot \H)(f(x))}{\cW (x )}\left(m - \frac{|\nabla 
			f(x)|^2}{\psi^2(f(x))}\right), 
		\end{aligned}
		\end{equation*}
where $\textnormal{Div}_M$ is the divergence taken with respect to the metric 
on $M$. 
\end{prop}


\section{Curvature estimates for graphs contained in a 
slab}\label{Sec-Est-Curv-Cont-Slab}

In this section, we establish curvature estimates for graphs contained in a 
slab of $M_\psi \times I$. We say that a graph $\Gamma_f$ of a function 
$f:M\to I$ is {\it 
	contained in a slab}, 
if, and only if, the function $f$ is bounded on $I$, \ie there 
exist $a,b\in I$ such that $f(x)\in [a,b]$ for all $x\in M$.


Our main ingredient to obtain these estimates is the 
following technical proposition. 

\begin{prop}\label{Prop-PrinCurvRefin}
	Let $M^m$ be a complete Riemannian manifold with sectional 
	curvature boun\-ded below and $\Gamma_f \subset M_\psi \times I$ a graph. 
	\begin{enumerate}
		\item If $f$ is bounded from below, then there is a 
	sequence $(x_n)\subset M$ such that the 
	principal curvatures 
	$\lambda_i$ of the graph of $f$ satisfies 
	\begin{equation}\label{Eq-PrinCurvEstpositive}
	\lambda_i(\phi(x_n)) < \frac{1}{n\psi(f(x_n))} +|\mathcal{H}(f(x_n))|, 
	\;\; \forall i=1,...,m,    
	\end{equation}
	for $n$ sufficiently large.
	Moreover, $f(x_n)\to \inf f$.
	
	\item If $f$ is bounded from above,   then 
	there 
	is a sequence $(y_n)\subset M$ such that  
	the 
	principal curvatures 
	$\lambda_i$ of the graph of $f$ satisfies 
	\begin{equation}\label{Eq-Princ.CurvEstnegative}
	-\lambda_i(\phi(y_n)) < \frac{1}{n\psi(f(y_n))} +|\mathcal{H}(f(y_n))|, 
	\;\; \forall i=1,...,m,   
	\end{equation}
	for $n$ sufficiently large. Moreover, $f(y_n)\to \sup f$.
\end{enumerate}
\end{prop}

In order to establish the proposition above we need to introduce some 
notation. 
Given $t_0\in I$ and $\sigma_0\in \R$, consider the 
function 
$\sigma: I \rightarrow J:=\sigma(I)\subset \R$  
defined by 
\begin{equation}\label{eq028}
	\sigma(t): = \sigma_0-\int^{t}_{t_0}\frac{1}{\psi(u)}du. 
\end{equation}
The map 
$\tau: M_{\psi}\times I \rightarrow M\times J\subset M\times \R$, defined by 
$\tau(x,t) := (x,\sigma(t))$,
is a conformal reversing orientation diffeomorphism. More 
precisely,
\begin{equation}\label{Eq-Tau-Conform}
	\langle 
	v,w\rangle_{\psi}=\psi^2(t) \langle 
	d\tau_{(x,t)}(v),d\tau_{(x,t)}(w)\rangle_{M\times\mathbb{R}},
\end{equation}
for all $v,w \in T_{(x,t)} (M\times I)$ (\cf \cite[Section 2.3]{AD1}), where 
$\langle \cdot, \cdot \rangle_{M\times \R}$ is the usual product metric. 
Define the function $\widehat{f} := \sigma\circ f$ and let $\widehat{\phi}$ 
be the canonical 
parametrization of $\Gamma_{\widehat{f}}$, the graph of $\widehat{f}$,
in 
$M\times\mathbb{R}$. Since
\begin{equation*}
	\widehat{\phi}(x) = \left(x,(\sigma\circ f)(x)\right) 
	= \tau(x,f(x)) = (\tau\circ \phi)(x),
\end{equation*}
it follows from  \Eqref{Eq-Tau-Conform} that   
\begin{equation}\label{Eq-InducedMetric}
	\langle d\widehat\phi(\cdot),d\widehat\phi(\cdot)\rangle_{M\times 
		\R} = \frac{1}{(\psi\circ 
		f)^2}\left\langle d\phi(\cdot) , d\phi(\cdot) 
	\right\rangle_{\psi}.
\end{equation}
We choose the following unit normal vector to the $\Gamma_{\widehat{f}}$
\begin{equation}\label{Eq-Normal-Graph-Product}
	\widehat{\eta}(\widehat{\phi}(x)):=\tau_{\ast}\left((\psi\circ f)(x)\cdot 
	\eta(\phi(x))\right)  = 
	\frac{1}{\widehat{\cW}(x)}\left( 
	-\widetilde{\nabla 
		\widehat{f}(x)} + \partial_t \left( \widehat{\phi}(x)\right)\right),
\end{equation} 
where 
\begin{equation}\label{eq002}
	\widehat{\cW}(x):=\sqrt{1+|\nabla \widehat{f}(x)|^2} = 
	\frac{\cW(x)}{\psi(f(x))}.
\end{equation}


Let $\Theta:\Gamma_f\to 
[-1,0)$ be the angle function $\Theta := \langle \eta,\partial_t\rangle$.
Using equations \eqref{wpprop01}, 
\eqref{Eq-Tau-Conform} and 
\eqref{Eq-InducedMetric}, the relation between $A$ and $\widehat{A}$, the 
second fundamental 
forms of $\Gamma_f$ in $M_{\psi}\times I$ and 
$\Gamma_{\widehat{f}}$ in $M\times \mathbb{R}$ 
with respect to the $\eta$ and $\widehat{\eta}$, is given by
\begin{equation}\label{eq012}
\widehat{A}( d\widehat{\phi}_x (v)) = (\psi\circ f)(x) \cdot 
d\tau_{\phi(x)} \Bigl( 
( A + ((\mathcal{H}\circ f) \cdot (\Theta\circ \phi))(x) \cdot Id 
)\left( 
d\phi_x(v)\right)\Bigr),
\end{equation} 
for all $x\in M$ and $v\in T_x M$, where $Id$ is the identity operator of
$T_{\phi(p)}\Gamma_f$. From this equation we easily deduce that
the principal curvatures $\lambda_i$ of $\Gamma_f$ in 
	$M_{\psi}\times I$ 
	and $\widehat \lambda_i$ of $\Gamma_{\widehat{f}}$ in $M\times 
	\mathbb{R}$ are  
	related by (see Guan and Spruck \cite[Section $2$]{GS} for similar 
	formulas) 
	\begin{equation}\label{Eq-PrinCurvWarpProd}
		\widehat\lambda_i\circ \widehat\phi= (\psi\circ f)\Bl 
		\lambda_i\circ {\phi} + 
		(\mathcal{H}\circ f)\cdot (\Theta\circ\phi)\Br.
	\end{equation}
	As a consequence, the mean curvatures $H$ of $\Gamma_f$ and $\widehat{H}$ 
	of $\Gamma_{\widehat{f}}$ satisfy 
	\begin{equation}\label{eq013}
		\widehat{H}\circ \widehat\phi = (\psi\circ f)\Bigl( H\circ\phi + 	
		(\mathcal{H}\circ f)\cdot (\Theta \circ \phi) 
		\Bigr).
	\end{equation}
%
%
%
%

Now we are in condition to present the proof of Proposition 
\ref{Prop-PrinCurvRefin}.

\begin{proof}[Proof of Proposition \ref{Prop-PrinCurvRefin}]
	Consider the function $\hat{f} = \sigma\circ f$ as defined before. 
	
	(i) First we assume that $f$ is bounded from below, so $\hat{f}$ is 
	bounded from above.
	Using the Omori-Yau 
	maximum principle \cite{OM,STY} there is a sequence $(x_n)\subset M $ 
	such 
	that 
	$\widehat{f}(x_n)\to \sup \widehat{f}$, $\|\nabla \widehat{f}(x_n)\| \to 
	0$ and
	\begin{equation*}
	\textnormal{Hess}\; \widehat{f}_{x_n}(v,v) < \frac{\|v\|^2}{n}, \qquad 
	\forall 
	v\in 
	T_{x_n} M, \;\; v\not = 0,
	\end{equation*}
	for $n$ sufficiently large. 
	Let $\widehat{\lambda}_i$ be the principal curvatures of 
	$\Gamma_{\widehat{f}}$ at $\widehat{\phi}(x_n)$ for $i=1,\ldots, m$. 
	Consider 
	$\{\widehat{e}_i\}_{i=1}^m$ be an orthonormal base of eigenvectors of 
	$\widehat{A}$
	associated to the principal curvatures. For each $i$ denote by $v_i\in 
	T_{x_n}M$ the unique vector such that $d\widehat{\phi}(x_n) (v_i) = 
	\widehat{e}_i$. 
	Then,
	\begin{equation*}
	\widehat\lambda_i(\widehat\phi(x_n) ) = \langle 
	\widehat{A}\widehat{e}_i,\widehat{e}_i\rangle_{M\times \mathbb{R}} = 
	\frac{\textnormal{Hess}\; 
	\widehat{f}_{x_n}(v_i,v_i)}{\widehat{W}(x_n)} 
	\leq 
	\frac{1}{n},\qquad 
	\forall i 
	=1\,\ldots,m.
	\end{equation*}
	Then from \Eqref{Eq-PrinCurvWarpProd} we obtain 
	\Eqref{Eq-PrinCurvEstpositive}. We observe that, since $\sigma$ is a 
	decreasing function $\sup \widehat{f} = 
	\sigma(\inf f)$, thus $f(x_n) \to \inf f$.
	
	\medskip
	
	(ii) To prove item (ii), let $\cM : I \to 
	-I$ be the 
	reverse 
	orientation diffeomorphism 
	$\cM(t) = -t$, where $-I = \setdef{t\in \R}{-t\in I}$, and 
	define 
	$\overline{\psi} = \psi \circ \cM^{-1}$.
	Consider the reverse orientation isometry (see \cite[pag. 123]{HM}),
	\begin{equation*}
		\varphi: M_{\psi} \times I \to M_{\overline{\psi}}\times 
		(-I),
	\end{equation*}
	given by $\varphi(x,t) = (x,\cM(t))$. Now define 
	$\overline{f}:=\cM\circ f$ and $\overline{\phi} := \varphi \circ 
	\phi$. As $\varphi$ is a reversing orientation isometry, the 
	principal curvature of $\Gamma_{\overline{f}}$ and $\Gamma_f$ with 
	respect to unit normal vector field (\ref{wp4}) are 
	related by $\overline{\lambda}_i(\overline{\phi}(x)) = - 
	\lambda_i(\phi(x))$.
	Thus the inequality \eqref{Eq-Princ.CurvEstnegative} is obtained using 
	the same ideas as in the first part applied to  
	$\overline{f} = \mathcal{M}\circ f$ in $M_{\overline{\psi}}\times (-I)$. 
\end{proof}

Using the Proposition \ref{Prop-PrinCurvRefin}, we obtain below a 
estimate for 
the mean curvature $H$ 
of a graph contained in a slab. Remember that the 
mean curvature $\H$ of a slice $M\times \{t\}$ is given by $(\ln \psi(t))'$, 
see 
\Eqref{Eq-MeanCurvSlice}.

\begin{theorem}\label{Theo-MeanCurvOmoriYau-f}
	Let $M$ be a complete Riemannian manifold with sectional curvature 
	boun\-ded below. 
	If $\Gamma_f\subset M_\psi \times \R$ is contained in a slab, then its mean 
	curvature $H$ satisfies 
	\begin{equation}\label{eq09}
	|\H(\inf f)|, |\H(\sup f)|\in \left[\inf_{\Gamma_f} |H|, 
	\;\sup_{\Gamma_f} 
	|H|\right].
	\end{equation}
	In particular,
	\begin{equation}\label{Eq-MeanCurvfBounded}
	\inf_{\Gamma_f} |H| \leq \sup_{f(M)} |\H|.
	\end{equation}
%
\end{theorem}

\begin{proof}
	By Proposition \ref{Prop-PrinCurvRefin}, there is 
	a sequence 
	$(x_n)\subset M$ such that $\widehat{f}(x_n) \to \sup \widehat{f}$, 
	and 
	$|\widehat{H}(\widehat{\phi}(x_n))| \to \inf |\widehat{H}| = 0$. 
	Since $\|\nabla \widehat f(x_n)\|\to 0$, we have that $\|\nabla 
	f(x_n)\|\to 
	0$.
	Using \Eqref{eq013} and taking the limit for $n\to\infty$, we obtain that
	\begin{equation}\label{Eq-EqualityMeanCurvGraphSlice}
	\begin{aligned}
	\lim_{n\to \infty} H(\phi(x_n)) = \lim_{n\to 
		\infty}\left(\frac{\widehat{H}(\widehat{\phi}(x_n))}{\psi(f(x_n))}  
	+ \frac{\psi(f(x_n))}{\cW(x_n)} \H(f(x_n))\right)
	= \H(\inf f).
	\end{aligned}
	\end{equation}
	Consequently,
	\begin{equation*}
	\begin{aligned}
	\inf_{\Gamma_f} |H| \leq 
	\lim_{n\to \infty} |H(\phi(x_n))| = |\H(\inf f)|\leq \sup_{\Gamma_f} |H|.
	\end{aligned}
	\end{equation*}

	As in the proof of Proposition 
	\ref{Prop-PrinCurvRefin} item (ii), 
	consider the function 
	$\overline{f}=\cM\circ f$ in the place of $f$. Note that, if $(y_n)$ is a 
	sequence such that $\overline{f}(y_n) \to \inf \overline{f}$ then 
	$f(y_n)\to\sup f$. Thus,
	using $\overline{f}$ instead of $f$ we obtain the other claim of 
	the	statement.
\end{proof}

\begin{remark}\label{Rem-1}
	When $f$ is only bounded from below (resp. above), it is possible to 
	obtain the 
	estimate \eqref{Eq-MeanCurvfBounded}, supposing that the mean curvature of 
	$\Gamma_f$ is non-negative 
	(resp. non-positive) and $\inf(\psi\circ f)>0$. In this case, we only have
	\begin{equation}\label{Eq-fBoundBel}
	\inf_{\Gamma_f} H \leq \H(\inf f)\qquad 
	\Bl\text{resp.}\;\;\inf_{\Gamma_f} |H| \leq 
	|\H(\sup 
	f)|\Br.
	\end{equation}
\end{remark}
\begin{remark}
	In  
	\cite[Theorem 2]{SAL2}, it is constructed 
	a bounded from below function, $f: 
	\mathbb{H}^2\to \R$, whose the graph $\Gamma_f\subset \mathbb{H}^2\times 
	\R$  has constant mean curvature $H= -1/2$, with 
	respect to the unit normal vector field $\eta$.
	This shows that the Remark \ref{Rem-1} is not 
	true when $f$ is
	bounded above and $H\geq 0$ (resp. $f$ bounded below and $H\leq 0$). 
\end{remark}

\begin{cor}\label{Corol-MeanCurvature-Slice}
	Let $M$ be a complete Riemannian manifold with sectional curvature 
	boun\-ded below. Suppose that no two slices have the same mean curvature. 
	Then the unique graphs with constant mean curvature contained in a slab  
	are 
	the slices.
\end{cor}

\begin{proof}
	Let $\Gamma_f$ be a graph contained in a slab with constant mean 
	curvature. By the last theorem $\H(\inf f) = \H(\sup f)$. Therefore, $f$ 
	is constant.
\end{proof}


In a similar way that the mean curvature of the slices has strong influence 
on 
the mean curvature of graphs, we will see that the scalar curvature and the 
norm of the shape operator of the later is also greatly influenced by the 
sectional curvature and the normal sectional curvature of the former.

The following lemma shows that, under appropriate conditions on the sectional 
curvature of the fiber,  it is possible to relate directly the scalar 
curvature of a graph in the warped product $M_{\psi}\times I$ with $2$-th 
mean curvature of the inclusion map. This result will be useful to prove our 
main theorems about 
the scalar curvature and the norm shape operator of a graph. 

\begin{lemma}\label{ecpprop3}
	Suppose there exist $\alpha\leq 0 $ and $\beta \geq 0$ such that
	\begin{equation}\label{eq08}
		\alpha \leq \cK\circ\phi +\H'\circ f \leq \beta,
	\end{equation}
	for all $p\in M$. Then  
	\begin{equation}\label{eq16}
		\alpha  \leq  R\circ \phi - 
		H_2\circ \phi -\cK^{\perp}\circ f\leq 
		\beta,  
	\end{equation}
	where $H_2$ is $2$-th mean curvature and $R$ is 
	the scalar curvature  of $\Gamma_f$ in 
	$M_{\psi}\times I$.
\end{lemma}

	\begin{proof}		
		Fix $p\in M$ and consider an orthonormal 
		basis $\{e_1, ...,e_m\}$ of $T_{\phi(p)} \Gamma_f$ such that 
		$A(e_i)= \lambda_i e_i$ for all $i=1,\ldots,m$, and let $v_i 
		= (d\phi_p)^{-1}(e_i)$. Denote by $\Re$, $\Re^{M}$ and 
		$\Re_\psi$ the curvature tensors of $\Gamma_f$, $M$ and 
		$M_{\psi}\times I$, respectively. It follows from the Gauss 
		equation that   
		\begin{equation}\label{ecp3}
		\begin{aligned}
		(m-1) \Ric_{\phi(p)} \left( e_j\right) &= \sum^{m}_{i=1,i\not = j} 
		\left\langle 
		\Re\left( e_i, e_j\right)e_j, e_i \right\rangle_{\psi} 
		\\
		&= \sum^{m}_{i=1, i\neq j} 
		\left\langle \Re_\psi\left( e_i, e_j\right)e_j, 
		e_i \right\rangle_\psi + \lambda_j(\phi(p))(mH-\lambda_j(\phi(p))),
		\end{aligned}
		\end{equation}
		where $\Ric_{\phi(p)}$ is the Ricci curvature of $\Gamma_f$ at 
		$\phi(p)$.
		For all $i\not = j$ we have from \Eqref{wp3} and Proposition 
		\ref{ecpprop2} 
		that 
		\begin{equation}\label{ecp22}
		\begin{aligned}
		\left\langle  \Re_\psi\left( e_i, e_j\right)e_j, e_i 
		\right\rangle_{\psi}  
		 &= 
		\left(\frac{K_M\left(v_i,v_j\right)}{\psi^2(f(p))} + 
		\H'(f(p))-\left(\frac{\psi''}{\psi}(f(p))\right)\right)|\widetilde{v}_i\wedge
		 \widetilde{v}_j|_\psi^2\\
		&
		-\left(\frac{\psi''}{\psi}(f(p))\right)\Big(|\widetilde{v}_i|_\psi^2\left\langle
		 \nabla f(p),v_j\right\rangle_M^2 + 
		 |\widetilde{v}_j|_\psi^2\left\langle 
		\nabla f(p),v_i\right\rangle_M^2\Big)\\
		& 
		+2\left(\frac{\psi''}{\psi}(f(p))\right)\langle \widetilde{v}_i, 
		\widetilde{v}_j\rangle_\psi\left\langle \nabla 
		f(p),v_i\right\rangle_M\left\langle \nabla f(p),v_j\right\rangle_M,
		\end{aligned} 
		\end{equation}
		where $|\widetilde{v_i}\wedge \widetilde{v_j}|_{\psi}^2 = 
		|\widetilde{v_i}|_\psi^2|\widetilde{v_j}|_\psi^2 - \left\langle 
		\widetilde{v_i},\widetilde{v_j}\right\rangle^2_\psi$. 
		By direct computation we have
		\begin{equation*}
		\begin{aligned}
		1=|{e_i}\wedge {e_j}|_{\psi}^2 
		&= |\widetilde{v}_i\wedge \widetilde{v}_j|_{\psi}^2 + 
		|\widetilde{v}_i|_\psi^2\left\langle \nabla f(p),v_j\right\rangle_M^2 
		+ 
		|\widetilde{v}_j|_\psi^2\left\langle \nabla 
		f(p),v_i\right\rangle_M^2\\
		& -2\langle \widetilde{v}_i, \widetilde{v}_j\rangle_\psi\langle 
		\nabla f(p), v_i\rangle_M \langle \nabla f(p), 
		v_j\rangle_M,
		\end{aligned}
		\end{equation*}
		which implies that 
		\begin{equation}\label{Eq-CurvCurvSecGraph}
		\left\langle  \Re_\psi\left( e_i, e_j\right)e_j, e_i 
		\right\rangle_{\psi} = 
		\left(\cK(\phi(p))(v_i,v_j) + 
		\H'(f(p))\right)|\widetilde{v}_i\wedge
		\widetilde{v}_j|_{\psi}^2
		-\left(\frac{\psi''}{\psi}(f(p))\right).
		\end{equation}
		From the hypothesis and \Eqref{Eq-NormalSecCurv}
		we obtain that
		\begin{equation*} 
			\alpha|\widetilde{v}_i\wedge \widetilde{v}_j|_\psi^2\leq 
			\left\langle 
			\Re_\psi\left( e_i, e_j\right)e_j, e_i 
			\right\rangle_\psi 
			-\cK^{\perp}(f(p)) \leq 
		 	\beta|\widetilde{v}_i\wedge \widetilde{v}_j|_\psi^2.
		\end{equation*}
		It follows from 
		the relation  
		$1 = |e_i|_\psi^2 = \psi^2(f(p))|v_i|_M^2 + \left\langle \nabla 
		f(p),v_i\right\rangle_M^2 $
		that  
		\begin{equation*}
			|\widetilde{v}_i\wedge \widetilde{v}_j|_\psi^2 = 
			\psi^4(p)\Bl|v_i|_M^2\cdot |v_j|_M^2 - \left\langle 
			v_i,v_j\right\rangle_M^2 \Br\leq 
			\psi^4(p)\cdot |v_i|_M^2\cdot |v_j|_M^2\leq 1.
		\end{equation*}
		Therefore, 
		\begin{equation}\label{ecp23}
		\alpha \leq\left\langle 
		\Re_\psi\left( e_i, e_j\right)e_j, e_i 
		\right\rangle_{\psi}-\cK^{\perp}(f(p)) \leq  
		\beta. 
		\end{equation}
		By Equations \eqref{ecp3} and \eqref{ecp23}, 
		\begin{equation}\label{ecp6}
		\alpha\leq 
		\Ric_{\phi(p)} \left( e_j\right) 
		-\frac{1}{(m-1)}\lambda_j(\phi(p))\left(mH-\lambda_j(\phi(p))\right)
		-\cK^{\perp}(f(p))\leq 
		\beta. 
		\end{equation}
		Now remembering that
		\begin{equation*}
		H_2(\phi(p)) := \frac{1}{m(m-1)} \sum_{i\not=j} \lambda_i(\phi(p)) 
		\lambda_j(\phi(p)),
		\end{equation*} 
		summing over $j$ from $1$ to $m$ on inequality \eqref{ecp6} and 
		dividing by $m$,
		we prove the result.
	\end{proof}
	We point out that the hypothesis on previous lemma holds for graphs in 
	the hyperbolic space 
	models 
	$\mathbb{R}^n_{e^t}\times\mathbb{R}$ 
	and $\mathbb{H}^n_{\cosh t}\times \mathbb{R}$ for any $\alpha\leq 0$, 
	$\beta\geq 
	0$. We will return to this discussion in Section \ref{Sec-Applications}.

	The hypothesis in Lemma \ref{ecpprop3} is similar to the 
	\textit{convergence 
	condition}, that was used by many authors  \cite{Mon,AGH,HJ,MI,AL,ADR} 
	in formulating criteria for a graph, or more generally 
	a hypersurface, in a warped product space to be a slice.

The next theorem present estimates for the scalar curvature $R$ of a graph 
contained in a slab.

\begin{theorem}\label{Teo-ScalCurv}
	Let $M$ be a complete Riemannian manifold with sectional 
	curvature 
	boun\-ded below and $\Gamma_f\subset M_\psi \times I$ a graph contained in 
	a slab. If there are $\alpha\leq 0$ and $\beta\geq 0$ 
	satisfying inequality \eqref{eq08} over $M$, then the 
	scalar curvature $R$ of $\Gamma_f$ satisfies 
		\begin{equation}\label{Eq-ScalCurvfbounded}
		\begin{aligned}
		\inf_{\Gamma_f} |R| \leq  \beta -\alpha + \min\left\{4\H^2(\inf 
		f)+|\cK^{\perp}(\inf f)|,4\H^2(\sup f) +|\cK^{\perp}(\sup f)|
		\right\}.
		\end{aligned}
		\end{equation}
\end{theorem}

\begin{proof}
	 We can assume that the scalar curvature of the graph $\Gamma_f$ of $f$ 
	 does never vanish, otherwise there is nothing to prove. 
	 Then, by continuity of $R$ and connectedness of $\Gamma_f$, $R$ is always 
	 positive or always negative. 
	 Moreover, from  Lemma \ref{ecpprop3}, 
	 \begin{equation*}
	 R\circ \phi\leq \beta + H_2\circ\phi  +\cK^{\perp}\circ f\qquad 
	 \textnormal{over} \qquad M.
	 \end{equation*}
	 We have three possibilities:
	 \begin{itemize}
	 	\item[{\bf i)}] $R > 0$ in $\Gamma_f$.
	 	
	 	\item[{\bf ii)}] $R<0$ and there exists $p_0 \in M$ such 
	 	that 
	 	\[
	 	\beta+H_2(\phi(p_0)) +\cK^{\perp}(f(p_0)) \geq 0.
	 	\]
	 	
	 	\item[{\bf iii)}] $R<0$ and $\beta+ H_2(\phi(x)) 
	 	+\cK^{\perp}(f(x))<0$ for all 
	 	$x\in M$.
	 \end{itemize}
	 
	 First consider the sequence $(x_n)$ 
	 from Proposition \ref{Prop-PrinCurvRefin}. In this case $f(x_n) \to \inf 
	 f$.
	 
	 Assuming {\bf i)}, applying the inequality \eqref{eq16} at a point 
	 $x_n$, we consider the limit $n\to \infty$ and obtain
	 \begin{equation*}
	 	|R|\leq \beta + \H^2(\inf f)+ \cK^\perp(\inf f).
	 \end{equation*}

	 Assuming {\bf ii)} we obtain, by Lemma \ref{ecpprop3}, 
	 \begin{equation*}
	 {\alpha}-\beta\leq {\alpha}  +\cK^{\perp}(f(p_0)) + H_2(\phi(p_0)) 
	 \leq R(\phi(p_0))<0,  
	 \end{equation*}
	 which implies 
	 \begin{equation*}
	 \inf_{\Gamma_f} |R| \leq |R(\phi(p_0))| = 
	 -R(\phi(p_0))\leq {\beta -\alpha}.  
	 \end{equation*}

	 Assuming {\bf iii)}. The proof of this case follows along the lines of 
	 the 
	 proof of \cite[Theorem 1.2]{FONT} (see \cite[Theorem 1.4]{FF}), we 
	 present only the main differences.
	 
	 If there is a subsequence $(x_{n_k})$ of $(x_n)$ such that 
	 $H_2(\phi(x_{n_k}))\geq 0$ for all $x_{n_k}$, as 
	 $R<0$ and $\alpha\leq 0$ one has 
	 \begin{equation*}
	 0<-R(\phi(x_{n_k}))\leq -\alpha  
	 -\cK^{\perp}(f(x_{n_k}))-H_2(\phi(x_{n_k}))\leq -\alpha 
	 -\cK^{\perp}(f(x_{n_k})),
	 \end{equation*}
	 and so, 
	 $\displaystyle \inf_{\Gamma_f}|R|\leq -\alpha +|\cK^{\perp}(\inf f)|$.
	 
	 On the other hand, if $H_2(\phi(x_n))<0$ for all $n$, we have principal 
	 curvatures of both 
	 signs 
	 at 
	 every point of the sequence $(x_n)\subset M$. Moreover, 
	 by our assumption $\displaystyle 
	 H_2(\phi(x_n))< -\cK^{\perp}(f(x_n))-{\beta}$. 
	 Denoting by  $l$ the 
	 number of 
	 negative principal 
	 curvatures at $\phi(x_n)$, $n$ fixed, we have  
	 \[
	 \lambda_1(\phi(x_n))\leq \ldots\leq \lambda_l(\phi(x_n))<0 \leq 
	 \lambda_{l+1}(\phi(x_n))\leq \ldots \leq 
	 \lambda_m(\phi(x_n)),
	 \]
	 where $1\leq l\leq m-1$. Then, 
	 \begin{equation*}
	 \frac{m(m-1)}{2} H_2(\phi(x_n)) 
	 \geq \sum_{\substack{i=1,...,l \\ j=l+1,...,m}} \lambda_i(\phi(x_n)) 
	 \lambda_j(\phi(x_n)),
	 \end{equation*}
	 and 
	 \begin{equation*}
	 \begin{aligned}
	 0>\frac{m(m-1)}{2}R(\phi(x_n))
	 &\geq \Bl mH(\phi(x_n)) - 
	 \sum^{m}_{i=l+1}\lambda_i(\phi(x_n))\Br\sum^{m}_{i=l+1}
	 \lambda_i(\phi(x_n)) \\
	 &+ \frac{m(m-1)}{2}\left( {\alpha}  +\cK^{\perp}(f(x_n))\right).
	 \end{aligned}
	 \end{equation*}
	 Hence,
\begin{equation}\label{Eq-InEscCurvProfTheo}
\begin{aligned}
\frac{m(m-1)}{2}\inf_{\Gamma_f} |R| 
&\leq \Bl 
m|H(\phi(x_n))|+ \sum^{m}_{i=l+1}\lambda_i(\phi(x_n))\Br 
\sum^{m}_{i=l+1}\lambda_i(\phi(x_n))\\
& -\frac{m(m-1)}{2}\left( {\alpha}  + \cK^{\perp}(f(x_n))\right). 
\end{aligned}
\end{equation}
Then, we apply Proposition \ref{Prop-PrinCurvRefin}, take the limit $n\to 
\infty$ and use
\Eqref{Eq-EqualityMeanCurvGraphSlice} to obtain that
	  \begin{equation*}
	 	\inf_{\Gamma_f} |R| 
	 	\leq 4\H^2(\inf f) -{\alpha}  + |\cK^{\perp}(\inf f)|.
	 \end{equation*}
	 Comparing the estimates from the three cases we conclude that
	 \begin{equation*}
	 	\begin{aligned}
	 		\inf_{\Gamma_f} |R| \leq  \beta -\alpha + 4\H^2(\inf 
	 		f) + |\cK^{\perp}(\inf f)|.
	 	\end{aligned}
	 \end{equation*}
 
	 Analogously, using the sequence $(y_n)$ of Proposition 
	 \ref{Prop-PrinCurvRefin} 
	 instead the sequence $(x_n)$ and following the same steps as before, we 
	 obtain 
	 \begin{equation*}
	 \inf_{\Gamma_f} |R|\leq  4\H^2(\sup f) + \beta -{\alpha}  + 
	 |\cK^{\perp}(\sup f)|. 		
	 \end{equation*} 
	 Then, the result follows.	
\end{proof}		
	
	\begin{remark}\label{Rem-2}
	 If we add the  hypothesis that $H$ does not change sign we can replace 
	 the 
	 boundedness hypothesis on $f$ by boundedness above or below. 
	 More precisely, 
	 if $f$ is only bounded below and the mean curvature 
	 $H$ of the graph does not change the sign we are able to prove that
	 \begin{equation*}
	 \inf_{\Gamma_f} |R| \leq 2 \H^2(\inf f) + \beta - \alpha + 
	 |\cK^\perp(\inf f)|.
	 \end{equation*} 
\end{remark}

Now we pass to estimate the norm of the shape of $\Gamma_f$.  

\begin{theorem}\label{Teo-EntGrapSecForm2}	
	Let $M^m$ be a complete Riemannian manifold with sectional curvature 
	bounded from below and $\Gamma_f\subset M_\psi \times I$ a graph contained 
	in a slab. Denote by $A$ the shape operator of $\Gamma_f$.
	\begin{itemize}
		\item[(i)] If $\dim M = m \geq 3$ and the inequality
		\begin{equation}\label{eq021} 
			\Ric_{\Gamma_f} - \inf_{f(M)} \cK^{\perp}< 
			\inf_{M} 
			(\cK \circ \phi+\H'\circ f)\leq 0,
		\end{equation}
		holds on $M$, 
		then 
		\begin{equation*}
			\inf_{\Gamma_f} |A| \leq 
			3(m-2) \cdot \min \{\mathcal{H}(\inf f)|,|\H(\sup f)|\}.
		\end{equation*}
		
		\item[(ii)]	If the mean curvature of the graph $H$ does not change 
		sign, then
	\begin{equation*}
			\inf_{\Gamma_f} |A| \leq m \cdot \min \{ |\H(\inf f)|,|\H(\sup 
			f)|\}.
	\end{equation*}
	\end{itemize} 
\end{theorem}

\begin{proof}	
	(i) Let us denote by $\lambda_1,\ldots,\lambda_m$ the principal 
	curvatures of 
	the graph $\Gamma_f$, labelled by the condition $\lambda_1\leq 
	\ldots\leq \lambda_m$. 
	From equations \eqref{ecp3}, \eqref{Eq-CurvCurvSecGraph} and inequality 
	\eqref{eq021} we obtain
	\begin{equation}\label{eq17}
	\lambda_i(mH - \lambda_i)<0, \hspace{0.8cm} i=1,\ldots ,m.
	\end{equation}
	The above inequality implies that $\lambda_j\neq 0$, $\forall 
	j$, and that there exist principal curvatures of both signs at every 
	point of $\Gamma_f$. Denoting by $l$ the number of negative principal 
	curvatures, we have at any point of $\Gamma_f$,
	\begin{equation}\label{eq020}
	\lambda_1\leq \ldots \leq \lambda_l<0< \lambda_{l+1}\leq \ldots \leq 
	\lambda_m.
	\end{equation}

	Changing $f$ in $M_{\psi}\times I$ to $\overline{f} = \mathcal{M}\circ f$ 
	in $M_{\overline{\psi}}\times (-I)$, if necessary, we can assume that 
	$2\leq l \leq m-1$. For each $i=1,\ldots,l$, one has by inequality
	\eqref{eq17} that
	\begin{equation*}
	\lambda_1 +\ldots +\widehat{\lambda}_i+\ldots+\lambda_l+\lambda_{l+1} 
	+\ldots+ 
	\lambda_m > 0,
	\end{equation*}
	where $\widehat{\cdot}$ means that we omit this principal curvature in 
	the sum.
	Thus,
	\begin{equation}\label{eq98}
	l\left( \lambda_{l+1}+\ldots+\lambda_m\right) > 
	(l-1)\sum^{l}_{i=1}|\lambda_i|.
	\end{equation}
	By Proposition \ref{Prop-PrinCurvRefin}, there exists a sequence 
	$(x_n)\subset 
	M$ satisfying the inequality \eqref{Eq-PrinCurvEstpositive}. Using this 
	information on inequality \eqref{eq98}, we obtain
	\begin{equation}\label{eq99}
	\sum^{l}_{i=1}|\lambda_i|<
	\frac{l(m-l)}{l-1}\left( \frac{1}{n\psi(f(x_n))} + |\H(f(x_n))|\right),
	\end{equation}
	and so   
	\begin{equation}\label{eq100}
	\sum^{m}_{i=1}|\lambda_i| < 
	\frac{(m-l)(2l-1)}{l-1}\left( \frac{1}{n\psi(f(x_n))} + |\H(f(x_n))|\right).
	\end{equation}
	Thus one has the following estimate for the square of 
	the norm of second fundamental form at $\phi(x_n)$:
	\begin{equation}\label{Eq-SquareSecForm}
	|A|^2\left(\phi(x_n)\right) \leq \left( 
	\sum^{m}_{i=1}|\lambda_i|\right)^2
	< \left[ \frac{(m-l)(2l-1)}{l-1}\left( 
	\frac{1}{n\psi(f(x_n))} + |\H(f(x_n))| \right)\right]^2.
	\end{equation}
	Using that $\frac{(m-l)(2l-1)}{l-1}$ is a decreasing function on $l$ one obtain 
	\begin{equation*}
	\inf_{\Gamma_f}|A| < 3(m-2)\left(\frac{1}{nc} + |\H(f(x_n))|\right), 
	\end{equation*}
	where $c=\inf\limits_M (\psi\circ f)$.
	By taking $n\to +\infty$, we conclude 
	that 
	\begin{equation}\label{Eq-Norm1}
	\inf_{\Gamma_f}|A| \leq 3(m-2)|\H(\inf f)|. 
	\end{equation}
	
	If we use the sequence $(y_n)$ of Proposition \ref{Prop-PrinCurvRefin} 
	instead of the sequence $(x_n)$ and following the steps above we reach at
	\begin{equation}\label{Eq-Norm2}
		\inf_{\Gamma_f}|A| \leq 3(m-2)|\H(\sup f)|. 
	\end{equation}
	With the equations \eqref{Eq-Norm1} and \eqref{Eq-Norm2} we prove the 
	item (i).
	
	\medskip
	
	(ii) We use similar ideas as in the proof of \cite[Theorem 
	1.7]{FONT}. Without loss of generality we can assume that $H\geq 0$. We 
	have two possibilities, or all principal curvatures are non-negative or 
	there are principal curvatures with negative values. First consider the 
	sequence $(x_n)$ of Proposition \ref{Prop-PrinCurvRefin}.
	
	If all principal curvatures are non-negative then
	\begin{equation*}
		|A|^2(\phi(x_n)) = \sum_{i=1}^m \lambda_i^2(\phi(x_n)) 
		\leq m\Bl\frac{1}{nc}+ |\H(f(x_n))| \Br^2,
	\end{equation*}
	where $c=\inf\limits_M (\psi\circ f)$.
	Which implies that
	\begin{equation*}
		\inf_{\Gamma_f} |A| \leq \sqrt{m} \cdot |\H(\inf f)|< m \cdot 
		|\H(\inf f)|.
	\end{equation*}
	
	Now suppose that there are principal curvatures with negative values.  
	Let $l$ be the number of negative principal curvatures such that
	\begin{equation*}
		\lambda_1(\phi(x_n))\leq \ldots \leq
		\lambda_l(\phi(x_n))<0 \leq \lambda_{l+1}(\phi(x_n))\leq \ldots \leq 
		\lambda_m(\phi(x_n)).
	\end{equation*}
 	Since $H\geq 0$ we have
 	\begin{equation*}
 		\sum_{i=1}^{l} |\lambda_i(\phi(x_n))| = - \sum_{i=1}^{l} 
 		\lambda_i(\phi(x_n)) \leq  \sum_{i=l+1}^{m} \lambda_i(\phi(x_n)) \leq 
 		(m-l)\Bl \frac{1}{nc} + |\H(f(x_n))|\Br,
 	\end{equation*}
 	where $c = \inf\limits_M (\psi\circ f)$.
 	Consequently,
 	\begin{equation*}
 		\begin{aligned}
 			|A(\phi(x_n))|^2 &= \sum_{i=1}^{l} 
 			\lambda_i(\phi(x_n))^2+\sum_{i=l+1}^{m} 
 			\lambda_i(\phi(x_n))^2\\
 			&\leq \Bl\sum_{i=1}^{l} 
 			|\lambda_i\phi(x_n)|\Br^2+\sum_{i=l+1}^{m} 
 			\lambda_i(\phi(x_n))^2\\
 			&\leq \Bl(m-l)^2+(m-l)\Br\Bl \frac{1}{nc} + |\H(f(x_n))|\Br^2\\
 			&\leq m(m-1)\Bl \frac{1}{nc} + |\H(f(x_n))|\Br^2
 			\leq m^2\Bl \frac{1}{nc} + |\H(f(x_n))|\Br^2.
 		\end{aligned}
 	\end{equation*}
 	From this inequality we obtain
 	\begin{equation*}
 		\inf_{\Gamma_f} |A| \leq m \cdot |\H(\inf f).|
 	\end{equation*}
 	
 	If we use the sequence $(y_n)$ of Proposition \ref{Prop-PrinCurvRefin} 
 	instead the sequence $(x_n)$, we obtain the following inequality
	 \begin{equation*}
		\inf_{\Gamma_f} |A| \leq m \cdot |\H(\sup f)|,
	\end{equation*}	
	and the item (ii) follows. 
\end{proof}

%


\section{Curvature estimates for graphs not necessarily contained in a 
slab} \label{Sec-Est-Curv--Not-Cont-Slab}

In this section, we will make curvature estimates for graphs $\Gamma_f 
\subset M_\psi \times I$ in the case where $f$ is unbounded. We will divide 
the discussion in two cases, depending on whether the warped function $\psi$ 
has reciprocal with finite area or not.

\subsection{Warped functions whose reciprocal has finite area}

The bound hypothesis on $f$ is not always achieve in a general situation. In 
this section we remove this condition. Our  methods presented in the last 
section are based in the Omori-Yau maximum principle that is used in the 
Riemannian product after we apply the function $\sigma$ (see \Eqref{eq028}) on 
$f$. Therefore it is natural to ask about the boundedness of $\sigma$ instead 
$f$. By the definition of $\sigma $, it is clear his relation with the warped 
function $\psi$. However, it is not easy to find suitable conditions on 
$\psi$ to obtain boundedness on $\sigma$. There 
are classical 
examples that satisfy the boundedness 
on $\sigma$ as the Hyperbolic spaces (see Section \ref{Sec-Applications}). 
Let us present some examples about the boundedness of the function $\sigma$.

\begin{example}\label{Ex-Sigmabounded}
	Consider $\psi(t) = \cosh t$, $\psi:\R \to \R$. Taking the constant 
	$\sigma_0= \pi/2$ one 
	has 
	\begin{equation*}
		\begin{aligned}
			\sigma(t) = \pi/2  - \int_{0}^{t}\frac{1}{\cosh u} du
			= \pi-2\arctan(e^t), 
		\end{aligned}
	\end{equation*}
	which is bounded on $\mathbb{R}$. 
\end{example}
\begin{example} 
	Consider $\psi(t) = e^t$, $\psi: \R \to \R$. Taking  $\sigma_0=1$ we 
	obtain the function  
	$\sigma(t) = 
	e^{-t}$ that is only bounded below.  
\end{example}

\begin{example}
	Let $\psi: (0,\infty)\to \R$ be the identity function $\psi(t)=t$. Then, 
taking $\sigma_0=0$, we obtain that $\sigma(t) = \ln t^{-1}$, thus $\sigma$ 
is an unbounded function.	
\end{example}

Motivated by these examples, we define.
\begin{dfn}
	We say that a function $\rho:I\subset \R \to (0,\infty)$ has 
	\nomepsi\ when $1/\rho \in L^{1}(I)$, \ie
	\begin{equation*}
		\int_I \frac{1}{\rho(x)} dx <\infty.
	\end{equation*}
\end{dfn}

If $\psi$ has \nomepsi, then the 
function $\hat f = \sigma \circ f$ is a bounded function for any function $f$ 
on $M$. 
This simple observation allows us obtain curvature estimates for any graph 
$\Gamma_f \subset M_\psi \times I$, similar to the estimates in the previous 
section. Below we will only state the results, once the proofs are similar.

The Proposition \ref{Prop-PrinCurvRefin} has an analogue when $\psi$ has 
\nomepsi. More 
precisely, if $\psi$ has \nomepsi, then there are 
		sequences $(x_n), (y_n)\subset M$  such that the 
		principal curvatures 
		$\lambda_i$ of the graph of $f$ satisfies 
		\begin{equation}\label{Eq-PrinCurvEstpositive-psi}
			\lambda_i(\phi(x_n)) < \frac{1}{n\psi(f(x_n))} 
			+|\mathcal{H}(f(x_n))|, 
			\;\; \forall i=1,...,m,    
		\end{equation}
		and
		\begin{equation}\label{Eq-Princ.CurvEstnegative-psi}
			-\lambda_i(\phi(y_n)) < \frac{1}{n\psi(f(y_n))} 
			+|\mathcal{H}(f(y_n))|, 
			\;\; \forall i=1,...,m,   
		\end{equation}
		for $n$ sufficiently large. As consequence, we obtain the following 
		estimate for the mean curvature $H$ 
of the graph $\Gamma_f$.

\begin{theorem}\label{Theo-MeanCurvOmoriYau-sigma}
	Let $M$ be a complete Riemannian manifold with sectional curvature 
	boun\-ded from below. Assume that $\inf (\psi \circ f)>0$. If $\psi$ has 
	\nomepsi, then
	\begin{equation}\label{Eq-MeanCurvSigBounded}
		\inf_{\Gamma_f} |H| \leq \sup_{f(M)} |\H|.
	\end{equation}
\end{theorem}

Using Lemma \ref{ecpprop3} we can estimate the scalar curvature $R$ of a 
graph $\Gamma_f$ non-necessarily contained in a slab.

\begin{theorem}\label{Teo-ScalCurv-sigma}
	Let $M$ be a complete Riemannian manifold with sectional curvature 
	boun\-ded below and assume there are $\alpha\leq 0$ and $\beta\geq 0$ 
	satisfying inequality \eqref{eq08} over $M$. If $\inf (\psi\circ f)>0$ 
	and $\psi$ has \nomepsi, then   
		\begin{equation}\label{Eq-ScalCurvEntGrap}
			\begin{aligned}
				\inf_{\Gamma_f} |R| &\leq \beta -\alpha+ 
				4\sup_{f(M)}|\mathcal{H}|^2  
				+  \sup_{f(M)}|\cK^{\perp}|.
			\end{aligned}
		\end{equation}
\end{theorem}

\begin{remark} \label{Rem-3}
	Even in the case that $\psi$ does not have \nomepsi, we can obtain 
	estimates for the mean curvature $H$ and the scalar curvature $R$ of a 
	graph. 
	If the function $\sigma$ (see \eqref{eq028}) is only 
	bounded from above (resp. below) we can obtain the 
	same estimate of \eqref{Eq-MeanCurvSigBounded} under the assumption that 
	the mean curvature of $\Gamma_f$ is non-negative (resp. non-positive).

	Similarly, if $\psi$ does not have \nomepsi, but the function $\sigma$ 
	is bounded from above or bounded from below, it is possible to prove the 
	following estimate 
	\begin{equation*}\label{Eq-ScalCurvEntGrap2}
		\begin{aligned}
			\inf_{\Gamma_f} |R| &\leq \beta -\alpha+ 
			2\sup_{f(M)}|\mathcal{H}^2|  
			+  \sup_{f(M)}|\cK^{\perp}|,
		\end{aligned}
	\end{equation*}
	under the assumption that 
	the mean curvature of $\Gamma_f$ is non-negative (resp. non-positive). 
\end{remark}

We now present estimates for the norm of the shape operator of a graph 
of an unbounded function. In 
Theorem \ref{Teo-EntGrapSecForm} below we relax the condition 
that $\psi$ has \nomepsi\ and assume only partial boundedness on $\sigma$ 
(see 
\eqref{eq028}).

\begin{theorem}\label{Teo-EntGrapSecForm}
	Let $M$ be a complete Riemannian manifold with sectional curvature 
	boun\-ded below and $\Gamma_f \subset M_\psi \times I$ a graph. Assume 
	that $\inf (\psi\circ f) >0$ and
	the function $\sigma$ is bounded from below or above.
	\begin{itemize}
		\item[(i)] If $\dim M = m \geq 3$ and the inequality \Eqref{eq021}
		holds on $M$, then the norm of the shape operator $A$ of $\Gamma_f$ 
		satisfies
		\begin{equation*}
			\inf_{\Gamma_f} |A| \leq 
			3(m-2) \cdot \sup_{f(M)} |\mathcal{H}|.
		\end{equation*}  
		
		\item[(ii)]	If the mean curvature $H$ of the graph does not change 
		sign, then the norm of the shape operator $A$ of $\Gamma_f$ satisfies 
		\begin{equation*}
			\inf_{\Gamma_f} |A| \leq m \cdot \sup_{f(M)}|\H|.
		\end{equation*}
	\end{itemize} 
\end{theorem} 

The following example shows that the assumption in the last theorem that the  
function $\sigma$ is bounded from below or above can not be dropped.
\begin{example}
	Fix $a\in \R$ and consider the hyperbolic model space given by the upper 
	half-space 
	$\mathbb{H}^2 = \{(x,y)\in\mathbb{R}^2: y>0\}$ with the metric 
	$\displaystyle ({1}/{y^2})(dx^2+dy^2)$. Define the function 
	$f:\mathbb{H}^2\to \R$, 
	by 
	$f(x,y)=a \ln y$, with $a>0$. The function $\sigma(t) =t$ is not bounded 
	and 
	$\H(t) = 0$ 
	for all $t$. The graph of 
	$f$ has principal curvatures $\displaystyle
	\lambda_1 = \frac{-a}{\sqrt{1+a^2}}$, 
	$\lambda_2 = 0$ 
	calculated with respect to $\eta$. Moreover, $\displaystyle |A| = 
	\frac{|a|}{\sqrt{1+a^2}}$ 
	\big(see \cite[Example $1$]{AHE} and \cite[Example $10$]{HEU}\big).  This 
	shows 
	that the mean curvature of the graph $H$ is constant and negative, the 
	mean curvature of the slices are trivial and $|A|>0$.
\end{example}


\subsection{Warped functions whose reciprocal has not finite area}

In Remarks \ref{Rem-1}, \ref{Rem-2} and \ref{Rem-3} we observed that we can 
obtain curvature estimates with partial boundedness hypotheses. If we are in 
the position that $f$ is unbounded, not even partially 
bounded, and 
$\psi$ has not \nomepsi, $\sigma$ is not even partially bounded, we still can 
obtain estimates for the curvature elements using local estimates that we 
describe bellow. The idea behind the proofs of the next results is to use the
equations 
\eqref{Eq-PrinCurvWarpProd}, \eqref{eq013} to compare the 
curvature of a graph in a 
warped product space with the curvature of a graph in a Riemannian product 
and then to apply the \cite[Proposition 
3.3]{FF}. We will omit the proofs here.

For every closed metric ball with 
radius $r$, $\overline{B}_r\subset M$,  
there are $p, q\in \overline{B}_r$ such that
\begin{equation}\label{eq023}
	\begin{aligned}
		-\lambda_i(\phi(q)),\;\lambda_i(\phi(p))\leq 
		\frac{\nu_c(r)}{\inf_{\overline{B}_r}(\psi\circ 
			f)} + \sup_{f( \overline{B}_r)}|\mathcal{H}|,
	\end{aligned}
\end{equation}
for all $i=1\ldots,m$, where  
$c = \inf_{\overline{B}_r} K_M$, $K_M$ denotes the sectional curvature 
of $M$ and
\begin{equation}\label{eq07}
	\nu_s(t): =\left\{
	\begin{array}{ll}
		t^{-1}, &s \geq 0, \; t>0,\\
		\sqrt{-s} \coth\left( t \; \sqrt{-s}\right),  & s<0,\; t>0.
	\end{array}
	\right.
\end{equation}
With these estimates and \cite[Theorem 1.1]{FF} we obtain
\begin{equation*}
	\inf_{\phi(\overline{B}_r)}|H| \leq 
	\frac{1}{m \inf_{\overline{B}_r} 
		(\psi\circ 
		f)} \left( (m-1)\nu_d(r) + \frac{1}{r}\right) + \sup_{f( 
		\overline{B}_r)}|\mathcal{H}|, 
\end{equation*}
where $d = \inf_{\overline{B}_r}\textnormal{Ric}_M$ and 
$\textnormal{Ric}_M$ is the 
Ricci curvature of $M$.

When the Ricci curvature of $M$ is bounded from below we can 
take the limit for $r\to \infty$ in the last estimative to obtain the 
following 
result. 
\begin{theorem}
	Let $M^m$ be a complete Riemannian manifold with Ricci curvature 
	boun\-ded 
	from  below by a constant $d$. If
	$\displaystyle\inf(\psi\circ f)>0$, then the mean curvature $H$ of the 
	graph $\Gamma_f\subset M_\psi \times \R$ satisfies
	\begin{equation}\label{Eq-infHlimitlocal}
		\inf_{\Gamma_f} |H| \leq \frac{1}{\inf_M(\psi\circ f)} 
		\frac{(m-1)}{m}\sqrt{-d} + 
		\sup_{f\left(M\right)}|\mathcal{H}|.
	\end{equation}
\end{theorem}
This theorem is an extension of \cite[Corollary 1.2]{FF} to
entire graphs in warped product spaces.

We conclude the section with comments about local estimates for the scalar 
curvature and the norm of the shape operator of a graph (compare 
with \cite[Theorems 1.4, 1.6, 1.8]{FF}). We use local 
estimate for the 
principal curvatures given by \eqref{eq023} and applying the same proof's 
strategies that we used in  Section \ref{Sec-Est-Curv-Cont-Slab}, we obtain 
the following local estimates.

\begin{equation*}
	\begin{aligned}
		\inf_{\phi(\overline{B}_r)} |R| &< 
		2\left(\sup_{\phi(\overline{B}_R)} |H| 
		+ 
		\frac{\nu_d(r)}{\inf_{\overline{B}_r}(\psi\circ f)}
		+ \sup_{f(\overline{B}_r)}|\mathcal{H}|\right)
		\left(\frac{\nu_d(r)}{\inf_{\overline{B}_r}(\psi\circ f)} + 
		\sup_{\phi(\overline{B}_r)}|\mathcal{H}| \right) \\
		&\qquad +\beta -\alpha+ \sup_{f(\overline{B}_r)}|\cK^{\perp}|,
	\end{aligned}
\end{equation*}
where $d=\inf_{\overline{B}_r} \Ric_M$, $\alpha\leq 0$ and 
$\beta\geq 0$ are constants such that \Eqref{eq08} holds over 
$\overline{B}_r$. 

For the norm of the shape operator we have two estimates. 
First, if $m\geq 3$ and  \Eqref{eq021} holds over 
$\overline{B}_r$
then 
\begin{equation*}
	\inf_{\phi(\overline{B}_r)} |A| \leq 
	3(m-2)\left(\frac{\nu_c(r)}{\inf_{\overline{B}_r}(\psi\circ f)}+ 
	\sup_{f(\overline{B}_r)} |\mathcal{H}|\right), 
\end{equation*}
where $K_M\geq c$ on ${\overline{B}_r}$. For the second estimate,
assume that $H$ does not change sign thus the norm of second fundamental 
form of $\Gamma_f$ satisfies 
\begin{equation*}
	\inf_{\phi(\overline{B}_r)} |A| \leq m\left( 
	\frac{\nu_c(r)}{\inf_{\overline{B}_r} (\psi\circ 
		f)} + 		
	\sup_{f(\overline{B}_r)}|\H|\right).
\end{equation*}
We observe that these estimates generalizes Theorems 1.2, 1.6 and 1.7 of 
\cite{FONT}, respectively.
If we 
have 
global hypothesis, then we can take the limit $r\to \infty$ on the 
inequalities 
above to obtain another global estimates for the scalar curvature and the 
norm of the shape operator.

\section{Pseudo-Hyperbolic Spaces and Space Forms} \label{Sec-Applications}

In this section we collect some consequences of the previous results.

\subsection{Pseudo-Hyperbolic spaces}
 
When $M^m$ is a 
complete Riemannian manifold and the warped function $\psi$ is either the 
exponential or 
the hyperbolic cosine, following the 
terminology introduced by Tashiro \cite[Section 3]{TASH}, we call the 
corresponding 
warped 
product space a {\it pseudo-hyperbolic space}. The geometric quantities of 
these spaces 
are presented in the table below.

\begin{table}[!ht]
\centering
\begin{tabular}{|c|c|c|c|c|}
	\hline  Space &
	 $\H(t)$ & $\cK(p,t)$ & $\cK^{\perp}(p,t)$ & $\sigma(t)$ \\ 
	\hline &&&&\\ 
	$M^m_{\cosh t} \times \R$ & $\tanh (t)$ & 
	$\displaystyle\frac{K_M}{\cosh^2(t)}$ & $-1$ 
	& $\pi - 2\arctan(e^t)$\\&&&& \\
	\hline &&&&\\
	$M^m_{e^t} \times \R$ & $1$ & $\displaystyle\frac{K_M}{e^{2t}}$ & $-1$ & 
	$e^{-t}$\\&&&&\\
	\hline 	
\end{tabular} 
	\caption{Curvature elements of the slices and the function 
	$\sigma$ on the 
	pseudo-hyperbolic spaces.}\label{Table-PHSpace}
\end{table}

An important problem in warped product spaces is to classify the graphs
with constant mean curvature. A natural question is whether the slices are the
only graphs with this property. The Corollaries \ref{cor4-1} and \ref{cor4}, 
that follow from  Table \ref{Table-PHSpace} and Theorems 
\ref{Theo-MeanCurvOmoriYau-f} and \ref{Theo-MeanCurvOmoriYau-sigma},
are related to this question on the particular case of pseudo-hyperbolic 
spaces (see 
\cite[Theorem 5.2]{CdL}).



%

\begin{cor}\label{cor4-1}
	Let $M$ be a complete 
	Riemannian 	manifold with sectional curvature bounded from below and 
	$\Gamma_f \subset M^m_{e^t}\times\mathbb{R}$ a graph.
	\begin{enumerate}
		\item If $\Gamma_f$ has constant mean curvature $H$ and is contained 
		in a 
		slab, then $H=1$. 
		
		\item If $\Gamma_f$ is bounded from below by a slice, then 
		\begin{equation*}
		\inf_{\Gamma_f} |H| \leq 1.
		\end{equation*}
		
		\item 
		If $\Gamma_f$ is contained in a slab, then its
		mean curvature $H$ is positive at some point (with orientation given 
		by \Eqref{wp4}).
	\end{enumerate}

\end{cor}
 The item  (iii) above is a consequence of 
 \Eqref{Eq-EqualityMeanCurvGraphSlice}.
 
 \begin{cor}\label{cor4}
 	Let $M^m$ be a complete Riemannian manifold with  sectional 
 	curvature bounded below and $\Gamma_f \subset M^m_{\cosh 
 	t}\times\mathbb{R}$ a graph. 
 	\begin{enumerate}
 		\item The mean curvature of $\Gamma_f$ satisfies 
 		\begin{equation}\label{Eq-PseudHyp}
 		\inf_{\Gamma_f}|H|\leq 1. 
 		\end{equation}
 		In particular, if $\Gamma_f$ has constant mean curvature, then $ 
 		|H|\leq 
 		1$.
 		
 		\item If $\Gamma_f$ is contained in a slab, then 
 		\begin{equation*}
 			\inf_{\Gamma_f} |H| \leq |\tanh (\inf f)| <1.
 		\end{equation*}
 		Moreover, if $\Gamma_f$ has constant mean 
 		curvature, then $\Gamma_f$ is a slice.
 	\end{enumerate}
 	
 \end{cor}

	We now focus our discussion in the special warped product spaces 
	$\mathbb{H}^m_{\cosh 
	t}\times \R$ and $\R^m_{e^t}\times \R$. It is well known that these 
	spaces are isometric to the 
	hyperbolic space $\mathbb{H}^{m+1}$ (see \cite[Example 4.3, p. 
	725]{Mon}).  The slices in  $\mathbb{R}^m_{e^t}\times \mathbb{R}$ are 
	horospheres and then its  
	mean curvature is equal to $1$ (see \cite[p. 
	512]{AD}). In the case of $\mathbb{H}^m_{\cosh t}\times 
	\mathbb{R}$, the slices are equidistant hypersurfaces
	(hyperspheres) and so its mean curvature is $\tanh t$. 
	
	The Corollary \ref{cor4} applied 
	to the model
	$\mathbb{H}^m_{\cosh t} \times \R$ shows that the slices are the only 
	graphs with constant mean curvature contained in a slab. 
	From  \cite[Section 2.3]{AD1}, \cite[Theorem 
	$3.4$]{KF} and \cite[Theorem $A$]{DCL}, 
	we know that the only graphs in 
	$\mathbb{R}^m_{e^t}\times\mathbb{R}\equiv\mathbb{H}^{m+1}$ 
	with constant mean curvature are the slices. 
	In particular, there are 
	no entire minimal graphs in such model.




As consequence of Theorem \ref{Teo-ScalCurv} one has the following result 
concerning the scalar curvature of graphs in the model 
$\mathbb{H}^m_{\cosh t}\times \R$. We observe that in this model 
\Eqref{eq08} holds 
with $\alpha=\beta =0$. 

\begin{cor}\label{cor3}
	The scalar curvature $R$ of $\Gamma_f\subset\mathbb{H}^m_{\cosh t}\times \R 
	\equiv\mathbb{H}^{m+1}$ satisfies 
	\begin{equation}\label{eq007}
	\inf_{\Gamma_f} |R| \leq 5. 
	\end{equation}
	Furthermore, when $H$ does not change sign, we obtain
	\begin{equation}\label{Eq-InfScalarHyperbolic}
	\inf_{\Gamma_f} |R| \leq 3. 
	\end{equation}
	
\end{cor}


From Theorem 
\ref{Teo-EntGrapSecForm} we obtain the following estimates for the norm of 
the 
shape operator of a graph.

\begin{cor}\label{cor-normAcosh}
	Suppose that $m\geq 3$. If 
	$\Ric_{\Gamma_f} < -1$, then the norm of the 
	shape operator of $\Gamma_f\subset \mathbb{H}^m_{\cosh t}\times \mathbb{R} \equiv 
	\mathbb{H}^{m+1}$ satisfies 
	\begin{equation}\label{Eq-EntireFormFund}
		\inf_{\Gamma_f} |A| \leq 3(m-2).
	\end{equation}

\end{cor}
%
\begin{cor}\label{Cor-SecFormRef}
	If the mean  curvature $H$ of a graph $\Gamma_f\subset 
	\mathbb{H}^m_{\cosh t}\times \mathbb{R} \equiv 
	\mathbb{H}^{m+1}$ does not change sign, then the norm of its 
	shape 
	operator satisfies 
	\begin{equation}\label{Eq-SecFormRef}
		\inf_{\Gamma_f} |A| \leq m.
	\end{equation}	
\end{cor}
\begin{remark}
With the same hypothesis, the estimates \eqref{Eq-InfScalarHyperbolic}, 
\eqref{Eq-EntireFormFund}, 
\eqref{Eq-SecFormRef} still hold for entire graphs in 
$\mathbb{R}^m_{e^t}\times 
\mathbb{R}
\equiv \mathbb{H}^{m+1}$, under the additional 
assumption that the graph is bounded from below by a slice.
\end{remark}

Another warped product model for the Hyperbolic space is construct in the 
following way.
Let $\S^m$ be the $m$-dimensional sphere with radius one, \ie 
$\S^m:=\setdef{x\in \R^{m+1}}{|x|=1}$.
The hyperbolic space $\mathbb{H}^{m+1}$ can also be seen as the model 
$(\S^m)_{\sinh t}\times 
(0,\infty)$ with metric
\begin{equation*}
	\langle\cdot,\cdot\rangle_{\mathbb{H}^{m+1}} = \sinh^2 t 
	\langle\cdot,\cdot\rangle_{\S^m} + dt^2.
\end{equation*}
 The slices have mean curvature and sectional 
curvature given by $\H(t)=\coth t$ and $\cK(p,t)= \csch^2(t)$, respectively. 
The normal sectional curvature of the slice is constant, $\cK^\perp(p,t)=-1$ 
and $\sigma(t) = \displaystyle\log \coth \frac{t}{2}$. Let $f:\S^m \to 
(0,\infty)$ 
be a smooth function. By \Eqref{Eq-EqualityMeanCurvGraphSlice} the supremum 
of the mean curvature of $\Gamma_f$ is always greater 
than $1$. Moreover, Corollary \ref{Corol-MeanCurvature-Slice} implies the 
following result.
\begin{cor}\label{Cor-Sphere-Hype}
	If $\Gamma_f\subset 
	(\S^m)_{\sinh t}\times 
	(0,\infty) \equiv \mathbb{H}^{m+1}$ has constant mean curvature, then 
	$\Gamma_f$ is a slice. In particular, there are no minimal graph in the 
	model 
	$(\S^m)_{\sinh t}\times (0,\infty) \equiv \mathbb{H}^{m+1}$.
\end{cor}

We observe that in model $\mathbb{H}^{m+1}$ described above, \Eqref{eq08} 
holds with $\alpha=\beta=0$. Then 
by Theorem 
\ref{Teo-ScalCurv}, one has the following estimate for the scalar 
curvature of the graph of a function $f:S^m\to (0,\infty)$ 
\begin{equation*}
\inf_{\Gamma_f} |R|\leq 4 \coth^2(\sup f) +1. 
\end{equation*}
When $m\geq 3$, Theorem 
\ref{Teo-EntGrapSecForm2} implies
\begin{equation*}
	\inf_{\Gamma_f} |A|\leq 3(m-2)\coth(\sup f).
\end{equation*}

\subsection{Spheres}

Consider the $(m+1)$-dimensional sphere $\S^{m+1}\subset\R^{m+2}$ with 
the induced metric
\begin{equation*}
\langle\cdot,\cdot\rangle_{\S^{m+1}} = \sin^2\theta_{m+1} 
\langle\cdot,\cdot\rangle_{\S^{m}}+d\theta_{m+1}^2 ,
\end{equation*}
where $\theta_{m+1}\in (0,\pi)$. Therefore, we can view 
$\S^{m+1}-\{p_N,p_S\}$ as 
the 
warped product $ 
(\S^{m})_{\sin\theta_{m+1}}  
\times (0,\pi)$, where $p_N = (0,\ldots,1)$ and $p_S = (0,\ldots,-1)$. The 
curvatures of the slices that we need to apply our results are given by
\begin{equation*}
	\begin{aligned}
		\H(\theta_{m+1}) = \cot \theta_{m+1}, \;\; \cK(p,\theta_{m+1}) = 
		\csc^2(\theta_{m+1}), \;\; \cK^{\perp}(p,\theta_{m+1}) 
		= 1.
	\end{aligned}
\end{equation*}
Note that $\sigma(\theta_{m+1}) = \log\cot 
\Bl\frac{\theta_{m+1}}{2}\Br$ is unbounded. 

Using the Corollary \ref{Corol-MeanCurvature-Slice} we obtain.

\begin{cor}
	Any graph $\Gamma_f\subset (S^{m})_{\sin\theta_{m+1} \times (0,\pi)}$ 
	with constant mean 
	curvature is a slice. 
\end{cor}

It is not difficult to see that the numbers $\alpha$ 
and $\beta$ of 
\Eqref{eq08} are null, then applying Theorem \ref{Teo-ScalCurv} and Theorem 
\ref{Teo-EntGrapSecForm2} we obtain 

\begin{cor}
	The scalar curvature of a graph $\Gamma_f \subset 
	(S^m)_{\sin\theta_{m+1}}\times (0,\pi) $ satisfies
	\begin{equation*}
			\inf_{\Gamma_f} |R|\leq 4 \min\left\{\cot^2(\inf 
		f),\cot^2(\sup f)\right\}+1.
	\end{equation*}
	Moreover, if $m\geq 3$ the norm of the shape operator of $\Gamma_f$ 
	satisfies the 
	following estimative
	\begin{equation*}
		\inf_{\Gamma_f} |A|\leq 3(m-2) \min\big\{|\cot(\inf f)|,|\cot{(\sup 
		f)}|\big\}.
	\end{equation*}
\end{cor}

\subsection{Euclidean Spaces}

The expression of the Euclidean metric in spherical coordinates 
\begin{equation*}
	\langle\cdot,\cdot\rangle_{\R^{m+1}} = t^2 
	\langle\cdot,\cdot\rangle_{S^{m}} + dt^2.
\end{equation*}
show that $\R^{m+1}-\{0\}$ can be identified with the warped product 
$(\S^{m})_{t} \times (0,\infty)$.

The mean curvature, the 
sectional curvature and the normal sectional curvature at 
$(p,t)\in 
\S^{m}\times 
\{t\}$ of the slices are given by, respectively,
\begin{equation*}
\H(t) = \frac{1}{t}, \qquad \cK(p,t) = \frac{1}{t^2},\qquad \cK^\perp(p,t) = 
0.
\end{equation*}
We observe that $\sigma(t) = -\log t$ is unbounded in this model.

It is clear from Corollary 
\ref{Corol-MeanCurvature-Slice} the following result.
\begin{cor}
The only graphs with constant mean 
curvature in $(S^m)_t\times (0,\infty)$ are the slices. Moreover, there are 
no minimal graphs in this 
model.	
\end{cor}
Note that the mean curvature of a graph in $(S^m)_t\times (0,\infty)$ cannot 
be negative 
since we always have a sequence in $\S^m$ along which the mean curvature of 
the graph converges to a positive number (see 
\Eqref{Eq-EqualityMeanCurvGraphSlice}). Since in this model \Eqref{eq08} 
holds with $\alpha =\beta =0$ 
we can apply Theorem \ref{Teo-ScalCurv} and item ii) from Theorem 
\ref{Teo-EntGrapSecForm2} to obtain the next corollary. 

\begin{cor}
	The scalar 
	curvature of $\Gamma_f \subset (S^m)_t\times(0,\infty)$ satisfies
	\begin{equation*}
		\inf_{\Gamma_f} |R| \leq \Bl\frac{2}{\sup f}\Br^2 .
	\end{equation*}
	Moreover, if the mean curvature of the graph of $f$ does not change sign, 
	then 
	\begin{equation*}
	H\geq 0 \qquad \text{and}\qquad \inf_{\Gamma_f}|A| \leq \frac{m}{\sup f}.
	\end{equation*}
\end{cor} 

\section*{Acknowledgments}
The authors are grateful to Francisco Fontenele from UFF(Brazil) for the 
encouragement, helpful conversations in developing and many valuable 
suggestions which improved the quality of this paper.

\bibliography{BarCosHar}
\bibliographystyle{amsalpha}

\end{document}